\documentclass[11pt]{amsart}

\usepackage{amsmath,amssymb,xypic}

\newcommand\defi[1]{{\sl #1}}

\newcommand\C{\mathbb{C}}
\newcommand\F{\mathbb{F}}
\newcommand\N{\mathbb{N}}
\newcommand\Q{\mathbb{Q}}

\newcommand\Z{\mathbb{Z}}

\newcommand\CC{\mathcal{C}}

\newcommand\GL{\mathrm{GL}}
\newcommand\Hom{\mathrm{Hom}}

\newcommand\gp[1]{\langle\,#1\,\rangle}

\newcommand\seq\subseteq
\newcommand\sub\subset
\newcommand\ssm\smallsetminus

\newcommand\onto\twoheadrightarrow

\newtheorem{theorem}{Theorem}
\newtheorem{cor}[theorem]{Corollary}
\newtheorem{prop}[theorem]{Proposition}

\begin{document}

\title{Coproducts of Finite Groups}
\author{Chris Hall}

\begin{abstract}
We show that for any pair of non-trivial finite groups, their coproduct in the category of finite groups is not representable.
\end{abstract}

\thanks{We gratefully acknowledge Bob Guralnick for helpful discussions on this topic.}

\maketitle


Given a category $\CC$ and a pair of objects $X_1,X_2\in\CC$, the coproduct $X_1\coprod X_2$ is \defi{representable in $\CC$} iff
$$
	\exists X\in\CC,\ 
	\exists(\iota_{X_1},\iota_{X_2})\in\Hom_\CC(X_1,X)\times\Hom_\CC(X_2,X)
$$
such that
$$
	\forall Y\in\CC,\ 
	\forall (f_1,f_2)\in\Hom_\CC(X_1,Y)\times\Hom_\CC(X_2,Y),\ 
	\exists! f\in\Hom_\CC(X,Y)
$$
making the diagram
$$
	\xymatrix{
	& X_1\ar[dl]_{\iota_{X_1}}\ar[dr]^{f_1} & \\
	X\ar@{-->}[rr]^{\exists!f} & & Y \\
	& X_2\ar[ul]^{\iota_{X_2}}\ar[ur]_{f_2} &
	}
$$
commute.

Let $G,H$ be finite groups.  If we regard them as members of the category of all groups, then $G\coprod H$ is representable: it is the free product $G*H$, and $G\to G*H$ and $H\to G*H$ are the canonical inclusions.  The purpose of this note is to prove the following theorem:

\begin{theorem}\label{thm}
If $G,H$ are non-trivial groups, then the coproduct $G\coprod H$ in the category of finite groups is not representable.
\end{theorem}

The key will be the following proposition:

\begin{prop}\label{prop:key}
Let $G,H$ be non-trivial groups, and let $g\in G\ssm\{1\}$ and $h\in H\ssm\{1\}$.  For every $m\geq 1$, then there exist a finite group $T_m$ and a homomorphism $q_m\colon G*H\to T_m$ such that $|\gp{q_m(gh)}|>m$.
\end{prop}

Before proving the proposition we show how it implies Theorem~\ref{thm}.

\begin{proof}[Proof of Theorem~\ref{thm}]
Let $F$ be a finite group and
$$
	\iota_G\colon G\to F
	\text{ and }
	\iota_H\colon H\to F
$$
be homomorphisms.  Let
$$
	g\in G\ssm\{1\}
	\text{ and }
	h\in H\ssm\{1\},
$$
let $m$ be the order of $\iota_G(g)\iota_H(h)$, and let
$$
	q_m\colon G*H\to T_m
$$
be a homomorphism such that $|\gp{q_m(gh)}|>m$ as in Proposition~\ref{prop:key}.  Let
$$
	f_G\colon G\to T_m
	\text{ and }
	f_H\colon H\to T_m
$$
be the respective compositions of
$$
	G\to G*H
	\text{ and }H\to G*H
$$
with $q_m$ so that
$$
	\xymatrix{
		& G\ar[dl]_{\iota_G}\ar[d]_{f_G}\ar[dr] \\
		F & T_m & G*H\ar[l]_{q_m} \\
		& H\ar[ul]^{\iota_H}\ar[u]^{f_H}\ar[ur] &
	}
$$
commutes.  Observe that
$$
	\{\,f\in\Hom(F,T_m) : f_G=\iota_Gf\text{ and }f_H=\iota_Hf\,\}
$$
is empty since $f_G(g)f_H(h)$ has order exceeding $m$ while
$$
	(f\iota_G(g))(f\iota_H(h))=f(\iota_G(g)\iota_H(h))
$$
has order $m$.  Therefore one cannot find a morphism $F\to T_m$ to complete the above diagram, and hence $G\coprod H$ is not representable in the category of finite groups.
\end{proof}

The proof of the proposition will occupy the remainder of this note.

\begin{proof}[Proof of Proposition~\ref{prop:key}]\ 

To start we recall a special case of a result of Marciniak:

\begin{theorem}\label{thm:marciniak}
If there exists a faithful representation $G\times H\to\GL_n(\Q)$, then there exists a faithful representation $G*H\to\GL_n(K)$ where $K=\Q(t)$.
\end{theorem}

\begin{proof}
See \cite{Marciniak}.
\end{proof}

Recall that a group is \defi{residually finite} iff every non-identity element is contained in the complement of a finite-index normal subgroup.

\begin{cor}\label{cor:residually-finite}
$G*H$ is residually finite.
\end{cor}

\begin{proof}
Let $\Q[G]$ and $\Q[H]$ be the respective group algebras of $G$ and $H$, and let $n=|G|+|H|$.  If $V=\Q[G]\oplus\Q[H]$, then $V\simeq\Q^n$ and there exists a canonical faithful representation $G\times H\to\GL_n(\Q)$.  Let
$$
	\rho\colon G*H\to\GL_n(K)
$$
be the corresponding representation given by Theorem~\ref{thm:marciniak}.  It is faithful and the image is finitely generated, so the corollary is a consequence of the fact that any finitely generated subgroup of $\GL_n(\C)$ is residually finite.  Rather than appeal to this general fact though, we prove directly that $G*H$ is residually finite.

Let $w\in G*H$ be a non-identity element.  We must show that there is a finite-index normal subgroup of $G*H$ whose complement contains $w$.

The set $\rho(G\cup H)$ is finite and contained in $\GL(K)$, so the least common multiple of the denominators of all the entries of the matrices in $\rho(G\cup H)$ is a non-zero polynomial $D\in\Q[t]$.  Moreover, the greatest common divisor of the numerators of all the entries of $\rho(w)-I_n$ is a non-zero polynomial $N\in\Q[t]$.

Let $Z\sub\Q$ be the finite set of zeros of $D\cdot N$ in $\Q$ and $U:=Q\ssm Z$.  Let $S$ be the polynomial ring $\Q[t][1/DN]$ and
\begin{equation}\label{eqn:S-to-K}
	\GL_n(S)\to\GL_n(K)
\end{equation}
be the homomorphism induced by the inclusion $S\to K$.

Observe that the elements of $\rho(G\cup H)$ all have finite order and generate $\rho(G*H)$.  Therefore the elements of $\det(\rho(G\cup H))$ are roots of unity and generate $\det(\rho(G*H))$, so the latter is contained in $S^\times$.  In particular, $\rho$ factors through \eqref{eqn:S-to-K}.

For each $s\in U$, let $\F_s$ be the quotient field $S/(t-s)S$ and $\rho_s$ be the composition
$$
	G*H\to\GL_n(S)\to\GL_n(\F_s)
$$
where the last homomorphism is induced by the quotient $S\to\F_s$.  Observe that $\F_s$ is (canonically) isomorphic to $\Q$, so we can regard $\rho_s$ as a homomorphism
$$
	\rho_s\colon G*H\to\GL_n(\Q).
$$
Observe also that the least common multiple of the denominators of the entries of all matrices in $\rho_s(G\cup H)$ is a positive integer $D_s\in\N$, and that $\det(\rho_s(G*H))\seq\Z[1/D_s]^\times$.  Therefore the image of $\rho_s$ is contained in the image of the natural homomorphism
$$
	\GL_n(\Z[1/D_s])\to\GL_n(\Q),
$$
and hence there is a homomorphism
$$
	\rho_{s,p}\colon G*H\to\GL_n(\Z/p)
$$
for each $p\nmid D_s$.  By construction, $\rho_{s,p}(w)\neq I_n$ and the kernel of $\rho_{s,p}$ has finite index in $G*H$.  In particular, the complement of the latter is a finite-index subgroup containing $w$, so $G*H$ is residually finite as claimed.
\end{proof}

Observe that $gh$ has infinite order in $G*H$.  Therefore, for each $m\geq 1$, the element $w_m:=(gh)^{m!}$ is not the identity, hence Corollary~\ref{cor:residually-finite} implies that there exists a finite-index normal subgroup $K_m\seq G*H$ whose complement contains $w_m$.

To complete the proof of the proposition, we let $T_m:=(G*H)/K_m$ and
$$
	q_m\colon G*H\to T_m
$$
be the canonical quotient.  By definition,
$$
	q_m(w^{m!})=q_m(w_m)\neq 1,
$$
hence $q_m(w)$ has order exceeding $m$ as claimed.
\end{proof}


\bibliographystyle{amsplain}
\bibliography{coproducts}


\end{document}